\documentclass[10pt]{amsart}
\usepackage{amssymb,amsrefs,amsmath,amsthm,graphicx,enumerate}
\usepackage[all,cmtip]{xy}
\usepackage{hyperref}
\usepackage{color}

\usepackage{geometry}
\newtheorem{thm}{Theorem}[section]
\newtheorem{prop}[thm]{Proposition}

\newtheorem{cor}[thm]{Corollary}

\theoremstyle{definition}
\newtheorem{definition}[thm]{Definition}
\newtheorem{example}[thm]{Example}
\theoremstyle{remark}
\newtheorem{remark}[thm]{Remark}
\numberwithin{equation}{section}

\newcommand{\Z}{\mathbb{Z}}  
\newcommand{\C}{\mathbb{C}}  
\newcommand{\A}{\mathcal{A}}
\newcommand{\G}{\Gamma}  

\begin{document}

\title{On Betti numbers of Milnor fiber of hyperplane arrangements}

\author{KaiHo Tommy Wong}
\author{Yun Su}
\address{Department of Mathematics, University of Wisconsin-Madison, 
Madison, WI 53706}
\email{wkhtommy@gmail.com}
\email{ysu@math.wisc.edu}

\begin{abstract}
We find a combinatorial upper bound for the first Betti number of the Milnor fiber for central hyperplane arrangements, which improves previous similar results in \cite{DM} and in \cite{CDO}. In particular, we obtain a combinatorial obstruction for trivial algebraic monodromy of the first homology of Milnor fiber. Calculations and comparisons to known examples in \cite{CS} will be provided.

\end{abstract}

\maketitle
\section{Introduction}

Let $\A$ be a central hyperplane arrangement in $\C^{n+1}$ and $H_i,i=1,2,...,d$ be the defining equations of the hyperplanes of $\A$. Let $f=\prod_i H_i$ and $M(\A)=\C^{n+1}\setminus \A$. There is a $C^{\infty}$-locally trivial fibration $$F\hookrightarrow M(\A) \xrightarrow{f} \C^*,$$ where $F$ is called the Milnor fiber \cite{SP}. $F$ can be identified as the affine hypersurface $\{f=1\}$ in $\C^{n+1}$.

Many open questions have been raised subject to $F$. For instance, it is known that the integral cohomology ring of $M(\A)$ is given by the Orlik-Soloman algebra, which only depends on the intersection lattice  $L(\A)$ of $\A$ \cite{OS}. It has been conjectured that the integral homology, or the characteristic polynomial, hence the Betti numbers, of $F$ are also determined by $L(\A)$.

There are active work on this conjecture in the case of $\C^3$ with controls of the local multiplicities of the 1-flats, complex line given by intersection of the hyperplane components. Assume that the multiplicities of the 1-flats are at most 3 \footnote{projective line arrangement with only double and triple points as singularities}. Libgober showed that the characteristic polynomial of $F$ can be computed explicitly using superabundance (\cite{EM}). Moreover the characteristic polynomial is combinatorially determined by $L(\A)$, by the work of Papadima and Suciu in \cite{3C}. We refer to \cite{HS} for a survey of the study of Milnor fiber of arrangements.

While the general conjecture is still open, in the case of $\C^3$, Massey obtained an upper bound for the Betti numbers of the Milnor fiber with the combinatoris in $L(\A)$ using vanishing cycles \cite{DM}. A better combinatorial bound is obtained in \cite{CDO} by Cohen, Dimca, and Orlik, using perverse sheaves theory. In this paper, we improve these previous results by method similar to those in \cite{EM}and \cite{TT}. Specifically, we prove the following theorem.

\begin{thm} Suppose $\A=\{H_i\}_{i=1}^d\subset \C^3$ is a central hyperplane arrangement. Let $F$ be the Milnor fiber of $\A$. For each 1-flat $L$, let $d_L$ denote the number of hyperplanes containing $L$. Then $$d-1\leq b_1(F) \leq d-1 + min_{H_i}(\sum_{L\in H_i}(d_L-2)((d_L,d)-1)),$$ where $(d_L,d)$ is the greatest common factor of $d_L$ and $d$. Moreover, the characteristic polynomial of the monodromy action on $H_1(F,\C)$ has the form $$(t-1)^{d-1}p(t),$$ where $p(t)$ divides $$gcd_{H_i}(\prod_{L\in H_i} (\frac{t^{(d_L,d)}-1}{t-1})^{d_L-2} ).$$ 
\end{thm}

The structure of this paper is as follows. In section 2, we survey some basic facts about the Euler characteristic of the Milnor fiber an arrangement in $\C^3$ and the monodromy of the first homology. The conjecture in the case of $\C^3$ reduces to the study of non-trivial eigenspaces of the first homology of the Milnor fiber. In section 3, the notion of Alexander modules will be introduced for hyperplane arrangement, as another perspective for homology of the Milnor fiber. A divisibility theorem for Alexander modules for arrangements will be stated and the essential parts of the proof will be recalled. Then we prove theorem 1.1 in section 4. Section 5 consists of calculations and comparisons to examples in \cite{CS}. In section 6, we introduce the notion of Milnor fiber with multiplicities, and extend our results to this generalized notion of Milnor fiber. In the last section, we compute an upper bound of the second Betti number of the Milnor fiber of central hyperplane arrangements in $\C^4$.

\section{Basics about Milnor Fiber of Hyperplane Arrangement in $\C^3$}

Let $\A$ be a central arrangement of hyperplanes in $\C^3$ and $H_i,i=1,2,...,d$ be the defining equations of the hyperplanes in $\A$. The Milnor fiber $F$ can be identified with the hypersurface $\{\prod_i H_i = 1\}\subset\C^3$. A 1-flat $L$ of $\A$ is formed by the intersection of two or more hyperplanes in $\A$. For each 1-flat $L$, let $d_L$ denote the number of hyperplanes in the arrangement which contain $L$.

Since $f=\prod_i H_i$ is homogeneous, there is a fibration, called the Milnor fibration, $$F\hookrightarrow M(\A) \xrightarrow{f} \C^*.$$ The geometric monodromy $h:F\rightarrow F$ of the fibration is given by  coordinate-wise multiplication of $\xi$, where $\xi = e^{2\pi i /d}$ is a primitive $d$-root of unity. This monodromy automorphism generates a cyclic $\Z_d$ free action on $F$. The quotient space $F/\Z_d$ can be identified to the projective complement $\mathbb{P}^2\setminus [\A]$, where $[\A]$ is the image of $\A$ in $\mathbb{P}^2$ under the Hopf bundle \cite{CS}. Moreover, the Hopf bundle factors through this regular $d$-fold covering $F\rightarrow F/\mathbb{Z}_d$. The covering map is the composition $F\hookrightarrow M(\A) \xrightarrow{\text{Hopf}} \mathbb{P}^2\setminus [\A]$.

Therefore, the Euler characteristic $\chi(F)$ equals to $d\cdot \chi(\mathbb{P}^2\setminus [\A])$ and can be explicitly computed using local multiplicities of the flats of $\A$.

\begin{prop}
$\chi(\mathbb{P}^2\setminus [\A]) = 3-2d+\sum_L(d_L-1)$
\end{prop}

\begin{proof}
Recall that Euler characteristic is additive in the complex algebraic setting. Let $[H_i],i=1,...,d$ be the projective lines in $[\A]$ and $a_i$ be the number of intersection points between lines on $[H_i]$. Denote these intersection points by $p$ and their multiplicities by $m_p$. The set $\{p\}$ of intersection points are in one to one correspondence with the set of 1-flat $L$ in the affine arrangement $\A$. The multiplicities $m_p$'s are in correspondence with $d_L$'s.

\begin{align*}
\chi(\mathbb{P}^2\setminus [\A])& = \chi(\mathbb{P}^2) - \chi([\A]) \\
                &= 3- \chi(\sqcup_i ([H_i]\setminus \text{intersection points})) - \chi(\text{intersection points})\\
                &= 3 - \sum_i (2-a_i) - \sum_p 1\\
                & =3 -2d +\sum_i a_i - \sum_p 1 \\
                & =3-2d + \sum_p m_p -\sum_p 1 \\
                & = 3-2d +\sum_L(d_L-1)
\end{align*}

\end{proof}

\begin{prop}(cf \cite{DM})
$b_2(F) = b_1(F)-(2d-1)(d-1)+d\sum_L(d_L-1)$
\end{prop}

\begin{remark}
Proposition 2.2 implies that the first and the second Betti numbers of $F$ differ by some known quantity, which only involve the degree of the arrangements and the multiplicities of the flats. It was proved by Massey in \cite{DM} using L$\acute{e}$ numbers. On the other hand, it is also a quick corollary of proposition 2.1.
\end{remark}

The (geometric) monodromy of $F$ has finite order $d$, so as the algebraic monodromy on the first homology $$h_*:H_1(F;\C)\rightarrow H_1(F;\C).$$ $h_*$ is diagonalizable and its eigenvalues are $d$-roots of unity. We will denote the dimension of the eigenspace $H_1(F)_{\xi^{k}}$ by $b_1(F)_{\xi^k}$. The homology Wang exact sequence $$...\rightarrow H_1(F)\xrightarrow{h_*-I} H_1(F) \rightarrow H_1(M(\A))\rightarrow H_0(F) \xrightarrow{0} ...$$ yields that $b_1(F)_1 =\dim H_1(M(\A);\C)- \dim H_0(F;\C) = d-1$. As a result, the study of $b_1(F)$ reduces to the study of the eigenspaces of $H_1(F)$ associated to non-trivial eigenvalues, which will be referred as non-trivial eigenspaces in this paper.

\section{Divisibility result for Alexander Modules}

In this section, we regard the Milnor fiber as an infinite cyclic cover of $M(\A)$, up to homotopy. We then can identify the homology of the Milnor fiber as an \textit{Alexander module} over a Laurent polynomial ring, whose order is called the \textit{Alexander polynomial}.

\begin{thm}(\cite{RR})
The first Alexander polynomial of $\A$ is the characteristic polynomial of the monodromy $h_*:H_1(F;\C)\rightarrow H_1(F;\C)$.
\end{thm}

The degree of the (first) Alexander polynomial is equal to $b_1(F)$. Moreover, the multiplicities of the roots of the Alexander polynomial, which are the eigenvalues of $h_*$, indicate the dimensions of the associated eigenspaces. Here we introduce the notion of Alexander modules for a central arrangement in $\C^{n+1}$, while the definition applies in general to path connected finite CW-complex with free first integral homology.

\subsection{Alexander module of $\A$} (\cite{MT})

Let $\A$ be a central hyperplane arrangement in $\C^{n+1}$ and $M(\A) = \C^{n+1}\setminus \A$. The first integral homology $H_1(M(\A); \Z)$ is isomorphic to $\Z^d$ and is generated by the homology class of the meridians at each hyperplane component in $\A$. By \cite{HS}, the Milnor fiber $F$ is homotopy equivalent to the regular infinite cyclic cover $M(\A)_{\infty}$ of $M(\A)$ defined by the linking number homomorphism $$\varepsilon:\pi_1(M(\A)) \xrightarrow{\text{ab}} H_1(M(\A);\Z)\xrightarrow{\phi} \Z,$$ where $\phi$ sends each generator of $H_1(M(\A);\Z)$ to $1\in \Z$. $\varepsilon$ defines a $\G:=\C[t^{\pm 1}]$ automorphism by multiplication with $t^{\varepsilon(\alpha)}$. Hence, the homology with local coefficients $H_i(M(\A);\G)$, called the \textit{Alexander module}, has the following $\G$-modules identification (\cite{AH}) $$H_i(M(\A);\G) \cong H_i(M(\A)_{\infty};\C) \cong H_i(F;\C).$$ 

Since $F$ is an complex $n$-dimensional affine hypersurface, it has a $n$-dimensional finite CW complex structure (\cite{DB}). $H_i(M(\A);\G)$ is a torsion $\G$-module and a finite dimensional $\C$-vector space. Since $\G$ is a principal ideal domain, the Alexander module has the identification $\G/\lambda_1\oplus ...\oplus \G/\lambda_m$. The order $\lambda_1\cdots \lambda_m$ of the  Alexander module, called the \textit{Alexander polynomial}, is defined up to units in $\G$. The Alexander polynomial is the same as the characteristic polynomial $det(h_*-tI)$ of the Milnor fiber (\cite{RR}). We have a remark regarding Alexander modules in general.

\begin{remark}\cite{DN}
Let $X\subset Y$ be finite CW-complexes and their first integral homology are free. Let $\varepsilon:\pi_1(Y)\twoheadrightarrow \Z$ and $\G=\C[t^{\pm 1}]$ be the local coefficient defined by $\varepsilon$. Then $H_j(X;i^*\G)\xrightarrow{i_*} H_j(Y;\G)$ is a $\G$-module homomorphism and $H_0(X;\G)\cong H_0(Y;\G)\cong \G/(t-1).$
\end{remark}

\subsection{Stratification of $\A$ and local Alexander polynomials} (\cite{EM},\cite{TT})

In this section, we introduce a natural stratification of $\A$ and define the notion of local polynomials.

Define two points in $\A$ to be equivalent if the two collections of hyperplanes in $\A$ containing each point coincide. The equivalence classes form stratification of $\A$ and can be explicitly described: if $S\in L(\A)$, then the corresponding strata $X_S$ is $$X_S = S\setminus \cup_{S\nleq H} H.$$

Let $T(H_1)$ be a regular tubular neighborhood of $H_1$ in $\C^{n+1}$ and $T=T(H_1)\setminus \A$. According to the stratification, $T$ can be decomposed into subspaces; each of which corresponds to exactly one strata in $H_1$, such that it trivially fibers over the strata. Denote $X_j^k$ as the $j$-th $k$ dimensional strata in $H_1$ and $Y_j^k$ as the associated subspace in the tube $T$. The fiber of the trivial fibration is the complement of $\A$ in a small complex disk $\mathbb{D}_j^{n+1-k}$, transversal to $X_j^k$. Moreover, $\mathbb{D}_j^{n+1-k} \cap \A$ is a central hyperplane arrangement in $\mathbb{D}_j^{n+1-k} \cong \C^{n+1-k}$. In short, the space $T$ has a decomposition $T= \cup_{k,j}Y_j^k$, where each $Y_j^k$ is the product of $X_j^k$ and its corresponding fiber.

\begin{example}
Let $\A$ be given by $\{xy=0\}\subset \C^2$. The only 0-strata is the origin and the complements of the origin in the $x$-axis and $y$-axis are the 1-stratas. Let $H_1=\{x=0\}$. Then the strata in $H_1$ are $X^0 = \{0\}$ and $X^1$ being the complement of the origin in the $y$-axis. $Y^0$ is the complement of $\A$ in a small ball at the origin. $Y^1$ is the complement of $\A$ in a small tube around the $y$-axis. $Y^1$ fibers over $X^1$ with fiber $\C^*$.
\end{example}

\begin{example}
Let $\A$ be a central hyperplane arrangement in $\C^3$. $X^2$ will be the complement of 1-flats in $H_1$. The $X_j^1$'s are the 1-flats on $H_1$, minus the origin. On each $X_j^1$, the fiber is the complement of $\A$ in a small complex 2-dimensional normal disk, which is an affine central line arrangement complement.
\end{example}

If the intersection $Y_{j_1}^{k_1} \cap Y_{j_2}^{k_2}$ is non-empty and $k_2 \geq k_1$, then this intersection is a product space of a submanifold $Z_{j_1,j_2}^{k_1,k_2}$ of $X_{j_2}^{k_2}$ and the same fiber from $Y_{j_2}^{k_2} \rightarrow X_{j_2}^{k_2}$.

\begin{definition}
For each $X_j^k$, the local Alexander module is defined as  $H_l(\mathbb{D}_j^{n+1-k} \setminus \mathbb{D}_j^{n+1-k} \cap \A;i^*\G).$
\end{definition}

Note that all local meridians in $\mathbb{D}_j^{n+1-k} \cap \A$ are mapped to the meridian around $H_1$ via the induced map on fundamental group by inclusion. The induced local system is compatible to each local linking number homomorphisms. As $\G$-modules, $$H_l(\mathbb{D}_j^{n+1-k} \setminus \mathbb{D}_j^{n+1-k} \cap \A;i^*\G)\cong H_l(F_{\text{local}};\C)$$ The local Alexander polynomials $\xi_{k,l}(t)$ are the same as the characteristics polynomials of the local Milnor fiber for $\mathbb{D}_j^{n+1-k} \cap \A$.

\subsection{Divisibility theorem of Alexander Module}

The following theorem is a previous work of the first author in \cite{TT}, rephrased in the context of Alexander modules of hyperplane arrangement. There are three main steps in the proof, where the first two will be recalled here. Please refer to \cite{TT} for details.

\begin{thm}(\cite{TT})
For $0 \leq i \leq n$, the $i$-th global Alexander polynomial $\Delta_i(\A)$ divides the product of the local Alexander polynomials $\xi_{k,l}(t)$ associated to the strata in one of the hyperplanes, say $H_1$, with:
\begin{itemize}
\item $n-i \leq k \leq n$
\item $3n-3k-2i \leq l \leq n-k$
\end{itemize}
\end{thm} 

\begin{proof} \textbf{Step 1:}

First prove that $H_i(T;i^*\G) \xrightarrow{i_*} H_i(M(\A);\G)$ is a $\G$-module epimorphism.

An Lefschetz hyperplane theorem argument yields the following results, similar as in \cite{RF} $$\pi_i(T) \xrightarrow{i_*} \pi_i(M(\A))$$ is an isomorphism for $i \leq n-1$ and is an epimorphism for $i=n$.It follows that $\pi_i(M(\A),T) =0$ for $i\leq n$. Hence, $M(\A)$ has the homotopy type of $T$ attached with some cells of dimension $n+1$ and the lower skeleton of $M(\A)$ and $T$ are homotopic. Moreover, the same is true for their coverings. As a result, there are isomorphisms of $\G$-modules $$C_i(T_{\infty};\C) \cong C_i(M(\A)_{\infty};\C)$$ for $i\leq n$, and an injection $$C_{n+1}(T_{\infty};\C)\hookrightarrow C_{n+1}(M(\A)_{\infty};\C),$$ which prove the first step.

\textbf{Step 2:}

The main tool to analyze $H_i(T;\G)$ is the homology Mayer-Vietoris spectral sequence (\cite{DB2},\cite{EM}) $$E_{p,q}^1 : \oplus_{\alpha} H_q(\cap_{j=0}^p Y_{t_{\alpha,j}}^{k_{\alpha,j}} ; i^*\G) \Rightarrow H_{p+q}(T;i^*\G),$$ with $k_{\alpha,0}<...<k_{\alpha,p}$.

Recall that  $\cap_{j=0}^p Y_{t_{\alpha,j}}^{k_{\alpha,j}} \simeq Z_{\alpha}^{k_p} \times (\mathbb{D}^{n+1-k_p} \setminus \A)$, where $Z_{\alpha}^{k_p}\subset X_{t_{\alpha,j}}^{k_{\alpha,j}}$ is a submanifold. The fact that $i^*\G|_{Z_{\alpha}^{k_p}}$ is the constant sheaf and the Kunneth formula of homology yield

\begin{align*}
H_q(\cap_{j=0}^p Y_{t_{\alpha,j}}^{k_{\alpha,j}} ; i^*\G) & \cong \oplus_{l=0}^q (H_{q-l}(Z^{2k_p};\C) \otimes H_l(\mathbb{D}^{n+1-k_p} \setminus\A ;i^*\G)) \\
& \cong \oplus_{l=0}^q (H_l(\mathbb{D}^{n+1-k_p}\setminus \A;i^*\G))^{b_{q-l}(Z^{2k_p})} \\
& \cong \oplus_{l=0}^q (H_l(F_{\text{local}};\C))^{b_{q-l}(Z^{2k})}
\end{align*}

From above calculation, we will focus on local homology with $l\leq n-k_p$ because the local Milnor fiber has the homotopy type of a ($n-k$) dimensional CW complex. In addition, we also have $q-l\leq 2k_p$, that is $i-p-l\leq 2k_p$. Since $p+1$ is the number of subspaces of $T$ involved in the calculation, $p\geq k_p$. As a result, the interesting local homology are the in range $i-3k \leq l \leq n-k$.

\textbf{Step 3:}
The final step is another application of the Lefschetz hyperplane theorem, which leads to the desired ranges $n-i \leq k \leq n$ and $3n-3k-2i \leq l \leq n-k$ (for details, please see \cite{TT}).

\end{proof}

\section{On Arrangements in $\C^3$}

In this section, we apply the divisibility theorem to a central hyperplane arrangement $\A=\{H_i\}_{i=1}^d\subset \C^3$.

\begin{thm} Suppose $\A=\{H_i\}_{i=1}^d\subset \C^3$ is a central hyperplane arrangement. Let $F$ be the Milnor fiber of $\A$. For each 1-flat $L$, let $d_L$ denote the number of hyperplanes containing $L$. Then $$d-1\leq b_1(F) \leq d-1 + min_{H_i}(\sum_{L\in H_i}(d_L-2)((d_L,d)-1)),$$ where $(d_L,d)$ is the greatest common factor of $d_L$ and $d$. Moreover, the characteristic polynomial of the Milnor monodromy action on $H_1(F,\C)$ has the form $$(t-1)^{d-1}p(t),$$ where $p(t)$ divides $$gcd_{H_i}(\prod_{L\in H_i} (\frac{t^{(d_L,d)}-1}{t-1})^{d_L-2} ).$$ 
\end{thm}

The below two corollaries of theorem 4.1 are previous known results by Massey and by Cohen, Dimca, and Orlik, aiming to get combinatorial upper bounds for $b_1(F)$.

\begin{cor}[Massey \cite{DM}]
Suppose $\A=\{H_i\}_{i=1}^d\subset \C^3$ is a central hyperplane arrangement. Let $F$ be the Milnor fiber of $\A$. For each 1-flat $L$, let $d_L$ denote the number of hyperplanes containing $L$. Then $$d-1\leq b_1(F) \leq d-1 + \sum_{L}(d_L-2)((d_L,d)-1)),$$ where $(d_L,d)$ is the greatest common factor of $d_L$ and $d$. Moreover, the characteristic polynomial of the Milnor monodromy action on $H_1(F,\C)$ has the form $$(t-1)^{d-1}p(t),$$ where $p(t)$ divides $$\prod_{L} (\frac{t^{(d_L,d)}-1}{t-1})^{d_L-2} .$$ 
\end{cor}

\begin{cor}[Cohen, Dimca, Orlik \cite{CDO}]
For each $H_i\subset \A$ and $k=1,...,d-1$, $$b_1(F)_{\xi^k}\leq \sum_{L\in H_i, (d,d_L)\neq 1}(d_L-2)$$
\end{cor}

An explicit example will be used to illustrate the proof of theorem 4.1, from which the general proof follows. First we state a well known formula which is essential to the proof.

\begin{thm}[\cite{AC}]
Let $C = \{x^d=y^d\} \subset \C^2$ be a line arrangement with $d$ lines all passing through the origin. Then the Alexander polynomial of $\C^2 \setminus C$ is $(t-1)(t^d-1)^{d-2}.$
\end{thm}

\subsection{Illustrating Example}

Let $f=xyz(x+y+z)$ and $\A=\{f=0\}\subset \C^3$. Denote $H_z=\{z=0\}$. Let $T=T(H_z)\setminus \A$. We have $H_1(T;i^*\G)\twoheadrightarrow H_1(M(\A);\G)$. The Mayer-Vietoris spectral sequence applied to $H_1(T;\G)$ yields $E_{0,1}^1\oplus E_{1,0}^1 \twoheadrightarrow H_1(T;i^*\G)$. The following picture shows all partitions $Y_j^k$ in $T$ and the table shows all the $\G$-modules summands in $E_{0,1}^1$ and $E_{1,0}^1$.

\includegraphics[scale=0.38]{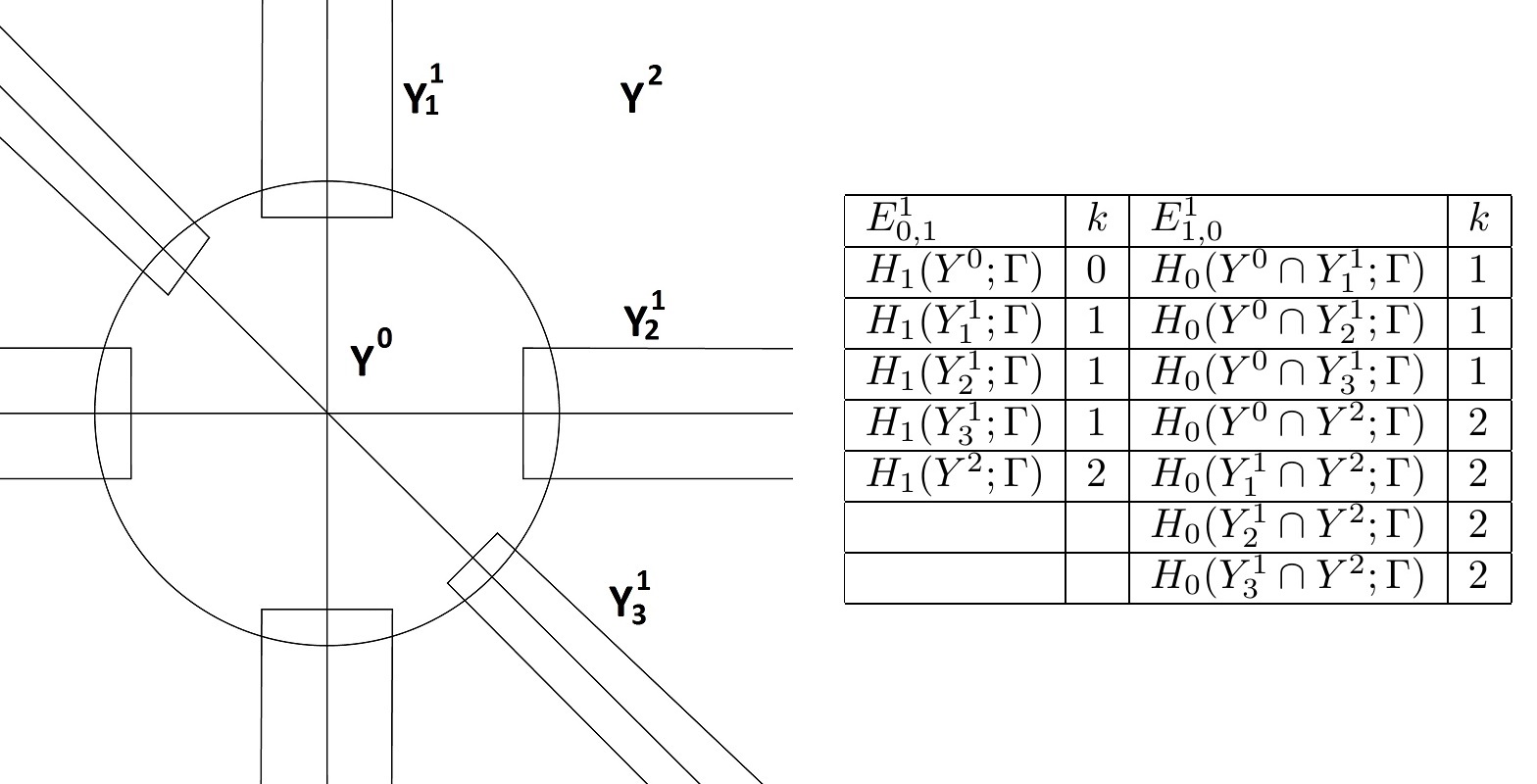}

By the restriction ranges in the divisibility result with $n=2, i=1$, only the summands with $k=2,l=0$ or with $k=1,l=1$ will contribute $H_1(T;i^*\G)$. Moreover, summands with $l=0$ are out of the consideration because they are isomorphic to a finite sum of $\G/(t-1)$, hence essentially only possibly contribute to the eigenspaces associated to the trivial eigenvalues.

Only the three terms in the form $H_1(Y_j^1;\G)$ are considered in the calculation of the upper bound for non-trivial eigenspaces. Each of them corresponds to a single 1-flat in $H_z$. Recall that $Y_j^1 \simeq X_j^1\times \C^2\setminus \{x^{d_j}=y^{d_j}\}$ and, in this case, $X_j^1 \simeq \C^*$. Therefore, $H_1(Y_j^1;i^*\G) \cong H_0(\C^2\setminus \{x^{d_j}=y^{d_j}\};i^*\G) \oplus H_1(\C^2\setminus \{x^{d_j}=y^{d_j}\};i^*\G)$.

Now write the first characteristic polynomial of $F$ as $(t-1)^{d-1}p(t)$, where $p(1)\neq 0$. So, $p(t)$ divides the product of the orders of $H_1(\C^2\setminus \{x^{d_j}=y^{d_j}\};i^*\G)$, which is, by theorem 4.4,  $$\prod_{j=1}^3(t-1)(t^{d_j}-1)^{d_j-2}.$$

Since $p(1)\neq0$ and that the characteristic polynomial of $F$ only has as $d$-roots of unity as zeroes, $p(t)$ divides  
$$\prod_{j=1}^3(\frac{t^{(d,d_j)}-1}{t-1})^{d_j-2}.$$

In this example, $d_1,d_2,d_3=2$, so $p(t)=1$. It is clear that the values $d_1,d_2.d_3$ are not used in the proof. Above argument actually holds in general.

\subsection{The general case}

For general central arrangement in $\C^3$, fix $H_1\subset \A$. Only the terms in the form $H_1(Y_j^1;\G)$ in the Mayer-Vietoris spectral sequence are considered in the calculation of the upper bound for non-trivial eigenspaces. Each  $Y_j^1 \simeq \C^* \times \C^2\setminus \{x^{d_j}=y^{d_j}\}$. Denote the first characteristic polynomial of $F$ by $(t-1)^{d-1}p(t)$, where $p(1)\neq 0$. Then $p(t)$ divides the order of $$\oplus_{L\in H_1} H_1(\C^2\setminus \{x^{d_L}=y^{d_L}\};\G).$$

As a result, $p(t)$ divides $\prod_{L\in H_1}(\frac{t^{(d,d_L)}-1}{t-1})^{d_L-2}.$ Since above is true for all hyperplane component in $\A$, $p(t)$ divides $$gcd_{H_i}(\prod_{L\in H_i}(\frac{t^{(d,d_L)}-1}{t-1})^{d_L-2}).$$

\section{Examples}

We now apply our results on specific examples. We try to detect whether the Milnor fiber monodromy is trivial. Most of the examples provided below have known Betti numbers computed by Fox calculus from the fundamental group of $M(\A)$ (\cite{CS}).

From the upper bound theorem, to obtain non-vanishing eigenspaces associated to non trivial eigenvalues, some local multiplicities (the $d_L$'s) must not be coprime to $d$ in all components. Certainly not all hyperplane arrangement would satisfy this condition.

\begin{example}
Let $\A$ be the $B_3$ arrangement with defining equation $$f=xyz(x-y)(x-z)(y-z)(x+y)(x+z)(y+z)$$
On the component $x=0$, the multiplicities of the 1-flats are 4,4,2,2. Note that $d=|\A|=9$ is coprime to the mattered multiplicities. By the upper bound theorem, $p(t)=1$. Hence, $H_1(F)=H_1(F)_1$ and $b_1(F)=9-1=8$. This result is confirmed by explicit calculations in \cite{CS}.
\end{example}

As suggested in \cite{EM}, non-vanishing eigenspaces with non-trivial eigenvalues only exist when the local multiplicities on different components have a nontrivial common divisor. This combinatorial obstruction can hardly met by many arrangements.

\begin{example}
Let $\A$ be the arrangement with defining equation $$f=xyz(x-y)(x+y)(x-z)$$
On the component $y=0$, the multiplicities of the 1-flats are 4,2,2. 
\newline On the component $z=0$, the multiplicities of the 1-flats are 3,2,2,2.

Then the upper bound theorem implies that $p(t)=1$ and that $h_*$ has no non-trivial eigenvalues $b_1(F)=b_1(F)_1=6-1=5$. As a comparison, Cohen, Dimca, and Orlik upper bound in \cite{CDO} gives that $b_1(F)_{\xi},b_1(F)_{\xi^2}\leq 1$, where $\xi$ is the primitive third root of unity.
\end{example}

The degree $d$ of $\A$ and all component-wise local multiplicities (besides 2) are not coprime in the following three examples. In addition, the local multiplicities are at most 3, one can use superabundance (\cite{EM}) or the combinatorial formula in \cite{3C} to compute them.

\begin{example}
Let $\A$ be the $A_3$ arrangement with defining equation $$f=xyz(x-y)(x-z)(y-z)$$
$b_1(F)=7$ and $b_1(F)_{\xi}=1=b_1(F)_{\xi^2}$, where $\xi$ is the primitive third root of unity. The multiplicities of the 1-flats in all components are 3,3,2, so the global Alexander polynomial of $\A$ divides $$(t-1)^5(\frac{t^3-1}{t-1})^2 = (t-1)^5(t^2-t+1)^2$$ 
Therefore, $b_1(F)\leq 9$ and $b_1(F)_{\xi},b_1(F)_{\xi^2} \leq 2$.
\end{example}

\begin{example}
Let $\A$ be the Pappus configuration $(9_3)_1$ arrangement with defining equation $$f=xyz(x-y)(y-z)(x-y-z)(2x+y+z)(2x+y-z)(-2x+5y-z)$$
$b_1(F)=10$ and $b_1(F)_{\xi}=1=b_1(F)_{\xi^2}$, where $\xi$ is the primitive third root of unity. The multiplicities of the 1-flats in all components is 3,3,3,2,2, so the global Alexander polynomial of $\A$ divides $$(t-1)^5(\frac{t^3-1}{t-1})^2 = (t-1)^5(t^2-t+1)^3$$  Hence, $b_1(F)\leq 14$ and $b_1(F)_{\xi},b_1(F)_{\xi^2} \leq 3$.
\end{example}

\begin{example}
Let $\A$ be the Pappus configuration $(9_3)_2$ arrangement with defining equation $$f=xyz(x+y)(y+z)(x+3Z)(x+2y+z)(x+2y+3z)(2x+3y+3z)$$
This is a counterexample showing that the upper bound theorem does not always detect trivial mondromy.
It is known that $b_1(F)=8$, and hence the mondromy is trivial. However, the multiplicities of the 1-flats in all components is 3,3,3,2,2. The upper bound theorem implies that the Alexander polynomial of $\A$ divides $$(t-1)^5(\frac{t^3-1}{t-1})^2 = (t-1)^5(t^2-t+1)^3,$$ which suggests possible room for non-trivial eigenspaces. 
\end{example}

Note that in the last three examples, since the multiplicities of the 1-flats do not change across hyperplane components, Cohen, Dimca, and Orlik upper bound gives the same result.

\section{On Milnor Fibers with Multiplicities}

We can extend the upper bound formula to the case of Milnor fiber with multiplicities for central arrangements in $\C^3$.

\begin{definition}
Let $\A=\{H_i\}_{i=1}^r\subset \C^{n+1}$ be a central hyperplane arrangement. Denote the defining equation of $H_i$ by $f_i$. Let $m_1,m_2,...,m_r$ be positive integers. Then $(\A,m)$ is called a multi-arrangement with defining equation $Q_m = \prod_{i=1}^r f_i^{m_i}.$ The corresponding Milnor fiber $F_m = Q_m^{-1}(1)$ is called the Milnor fiber with multiplicities $m= (m_1,...,m_r).$

\end{definition} 

\begin{prop}[cf \cite{TT}]
Let $(\A,m)$ be a multi-arrangement in $\C^{n+1}$. Define $$\varepsilon_m: \pi_1(M(\A)) \xrightarrow{\text{ab}} H_1(M(\A);\Z)\rightarrow \Z,$$ sending the meridian about $H_i$ to $m_i$. Then $\varepsilon_m$ defines a local coefficients $\G_m:=\C[t^{\pm 1}]$; and as $\G_m$-modules, $$H_*(M(\A);\G_m) \cong H_*(F_m;\C).$$
\end{prop}

\begin{remark}
The notion of $H_*(M(\A);\G_m)$ is a special case of twisted Alexander modules, which are studied intensively in \cite{TT}.
\end{remark}

\begin{thm}(theorem 2.8 in \cite{TT})
Let $(\A,m)$ be a multi-arrangement in $\C^2$. The first characteristic polynomial of $F_m$ is $gcd\{t^{m_i}-1\}_{i=1}^r \cdot (t^{m_1+...+m_r}-1)^{r-2}.$
\end{thm}

\begin{thm}
Let $(\A,m)$ be a multi-arrangement in $\C^3$. Let $d_L$ be the number of hyperplane in $\A$ that contains the 1-flat $L$. Denote the number of 1-flats in $H_i$ by $p_i$. The first characteristic polynomial of $F_m$ divides $$gcd_{i=1,...,r}[(t^{m_i}-1)^{2p_i+1}\cdot\prod_{L\in H_i}(\gcd\{t^m-1\}_{m\in L} \cdot (t^{\sum_{m\in L}m}-1)^{d_L-2})].$$
\end{thm}

\begin{proof}

The divisibility theorem holds for Alexander modules with $\G_m$-coefficients. In particular, $$H_1(T(H_1)\setminus \A;\G_m) \twoheadrightarrow H_1(M(\A);\G_m)$$ and the Mayer-Vietoris spectral sequence with $\G_m$-coefficients implies that  
$E_{0,1}^1\oplus E_{1,0}^1 \twoheadrightarrow H_1(T;i^*\G_m)$. By the divisibility result, $H_1(M(\A);i^*\G_m)$ is bounded above by the direct sum of local modules from $H_1$, with $(k,l) = (2,0),(1,1)$. Let us study the summands in the $E^1$ page of the corresponding spectral sequence:

\begin{itemize}
\item $H_0(Y_0\cap Y_j^1;\G_m)$ : not considered because $k=1,l=0$
\item \textcolor{red}{$H_0(Y_0\cap Y^2;\G_m)\cong \G/(t^{m_1}-1)$}
\item \textcolor{red}{$H_0(Y_j^1\cap Y^2;\G_m)\cong \G/(t^{m_1}-1)$}
\item $H_1(Y_0;\G_m)$ : not considered because $k=0$
\item \textcolor{red}{$H_1(Y_j^1;\G_m)$}
\item \textcolor{red}{$H_1(Y^2;\G_m)$}
\end{itemize}

The proof is completed by the following two computations. By theorem 6.4, the order of $H_1(Y_j^1;\G_m)$ is $$\gcd\{t^m-1\}_{m\in L} \cdot (t^{\sum_{m\in L}m}-1)^{d_L-2}),$$ where $L$ is the 1-flat sitting in $Y_j^1$. On the other hand,
\begin{align*}
H_1(Y^2;\G_m) & \cong (H_0(X^2;\C)\otimes H_1(\C^*;\G_m))\oplus (H_1(X^2;\C)\otimes H_0(\C^*;\G_m)) \\
& \cong 0 \oplus (\C^{b_1(X^2)}\otimes \G/(t^{m_1}-1))\\
& \cong (\G/(t^{m_1}-1))^{p_1}
\end{align*}

\end{proof}

\section{On Arrangements in $\C^4$}

In this section, we demonstrate such a Mayer-Vietoris argument shown in this paper can be applied to obtain combinatorial upper bound for the second Betti numbers of the Milnor fiber of central hyperplane arrangement in $\C^4$. In principle, bounds for higher Betti numbers can obtained through similar way.

\begin{prop} Suppose $\A=\{H_i\}_{i=1}^d\subset \C^4$ is a central hyperplane arrangement. For each flat $L$ of rank 1, denote $d_L$ be the number of hyperplanes which contain $L$ and $p_L$ be the number of 2-flats which are contained in $L$.  Let $F$ be the Milnor fiber of $\A$.  Then $$\chi(F) = 4d-3d^2+d\cdot\text{\# of 2-flats} + d\sum_{L}(3d_L-d_L^2+p_L-2)$$
\end{prop}

\begin{thm} Suppose $\A=\{H_i\}_{i=1}^d\subset \C^4$ is a central hyperplane arrangement. Let $F$ be the Milnor fiber of $\A$. Then $b_3(F) \leq 2-d-\chi(F)+b_2(F)$. For each 1-flat $L$, denote $d_L$ be the number of hyperplanes which contain $L$ and $p_L$ be the number of 2-flats which are contained in $L$.  $$b_2(F) \leq b_2(\C^4\setminus \A) + min_{H_i}(\sum_{L\in H_i}(B_L+(2p_L+1)(d_L-2)((d,d_L)-1)),$$ where $B_L$ is the sum of upper bounds of the second Betti numbers of the Milnor fiber of local central hyperplane arrangements for the 1-flats in $L$. 
\end{thm}

\begin{proof}

Since $b_1(F)\geq d-1$, it is clear that $b_3(F) \leq 2-d-\chi(F)+b_2(F)$.

By the Wang sequence: $$...\rightarrow H_i(F)\xrightarrow{h_*-I} H_i(F) \rightarrow H_i(M(\A)) \rightarrow H_{i-1}(F)\rightarrow...$$

Then $$\dim H_2(F;\C)_1 = \dim H_2(M(\A);\C) - \dim H_1(F;\C)_{\neq 1} \leq b_2(M(\A)).$$

We find an upper bound for the non-trivial eigenspaces of $H_2(F;\C)$, using a Mayer-Vietoris argument. We only consider the summands in the Mayer-Vietoris spectral sequence which possibly contribute the non-trivial eigenspaces. By the divisibility result of Alexander modules, $H_2(M(\A);\G)$ is bounded above by the direct sum of local Alexander modules from only one component, say $H$, of $\A$, with $(k,l) = (3,0),(2,0),(2,1),(1,2)$. Now analyze the summands in the $E^1$ page of the corresponding spectral sequence:

\begin{itemize}
\item $E_{2,0}$:   $l=0$
\item $E_{1,1}$: 
\begin{enumerate}
\item $H_1(Y^0 \cap Y_j^1;\G)$:   $k=1,l\neq 2$
\item \textcolor{red}{$H_1(Y^0 \cap Y_j^2;\G)$:   will only consider the summand with $k=2,l=1$}
\item $H_1(Y^0 \cap Y_j^3;\G)$:   $k=3$, only consider summand with $l=0$
\item \textcolor{red}{$H_1(Y_i^1 \cap Y_j^2;\G)$:   will only consider the summand with $k=2,l=1$}
\item $H_1(Y_i^1 \cap Y^3;\G)$:   $k=3$, only consider summand with $l=0$
\item $H_1(Y_j^2 \cap Y^3;\G)$:   $k=3$, only consider summand with $l=0$
\end{enumerate}

\item $E_{0,2}$: 
\begin{enumerate}
\item $H_2(Y^0 ;\G)$:   $k=0$
\item \textcolor{red}{$H_2(Y_i^1;\G)$:   will only consider the summand with $k=1,l=2$}
\item \textcolor{red}{$H_2(Y_j^2 ;\G)$:   will only consider the summand with $k=2,l=1$}
\item $H_2(Y^3 ;\G)$:   $k=3$, only consider summand with $l=0$
\end{enumerate}

\end{itemize}

Only the four terms in red are involved in the upper bound of non-trivial eigenspaces.

\begin{itemize}

\item The (2,1) part of $H_1(Y^0 \cap Y_j^2;\G)$ is isomorphic to  $$H_1(\C^2\setminus \{x^{d_{L_j}}=y^{d_{L_j}}\};\G)^{b_0(Z_j^{0,2})}\cong \textcolor{red}{H_1(\C^2\setminus \{x^{d_{L_j}}=y^{d_{L_j}}\};\G)}$$

\item The (2,1) part of $H_1(Y_i^1 \cap Y_j^2;\G)$ is isomorphic to  $$H_1(\C^2\setminus \{x^{d_{L_j}}=y^{d_{L_j}}\};\G)^{b_0(Z_{i,j}^{1,2})}\cong H_1(\C^2\setminus \{x^{d_{L_j}}=y^{d_{L_j}}\};\G)$$ Since there are $p_{L_j}$ 2-flats in $L_j$, we obtain $$\textcolor{red}{H_1(\C^2\setminus \{x^{d_{L_j}}=y^{d_{L_j}}\};\G)^{p_{L_j}}}$$

\item The (2,1) part of $H_2(Y_j^2;\G)$ is isomorphic to  $$H_1(\C^2\setminus \{x^{d_{L_j}}=y^{d_{L_j}}\};\G)^{b_1(X_j^2)} \cong H_1(\C^2\setminus \{x^{d_{L_j}}=y^{d_{L_j}}\};\G)^{b_1(L_j\setminus \text{2-flats})}$$ Since $L_j\setminus \text{2-flats}$ is a central arrangement complement with $p_{L_j}$ components, we obtain $$\textcolor{red}{H_1(\C^2\setminus \{x^{d_{L_j}}=y^{d_{L_j}}\};\G)^{p_{L_j}}}$$

\item The (1,2) part of $H_2(Y_j^1;\G)$ is isomorphic to  $$H_2(\C^3\setminus \text{local central arrangement at the j-th 2-flat};\G)^{b_0(X_j^1)}$$ We obtain $\textcolor{red}{B_L}$, by summing up over all 2-flats in $H$.

\end{itemize}

Multiplicities of non-trivial roots of the the orders from above three terms yield the formula $$b_2(F) \leq b_2(\C^4\setminus \A) + min_{H_i}(\sum_{L\in H_i}(B_L+(2p_L+1)(d_L-2)((d,d_L)-1)))$$

\end{proof}

\end{document}